\documentclass{article}

\usepackage{arxiv}

\usepackage[utf8]{inputenc} % allow utf-8 input
\usepackage[T1]{fontenc}    % use 8-bit T1 fonts
\usepackage{hyperref}       % hyperlinks
\usepackage{amsthm} % Used for creating new theorem and remark environments

\usepackage{url}            % simple URL typesetting
\usepackage{booktabs}       % professional-quality tables
\usepackage{amsfonts}       % blackboard math symbols
\usepackage{nicefrac}       % compact symbols for 1/2, etc.
\usepackage{microtype}      % microtypography
\usepackage{lipsum}
\usepackage{graphicx}
\graphicspath{ {./images/} }

\usepackage{mathtools}
\usepackage{amssymb}

\usepackage[capitalize,nameinlink,noabbrev]{cleveref} % To add \cref

\newtheorem{theorem}{Theorem}[section]
\newtheorem{definition}[theorem]{Definition}
\newtheorem{proposition}[theorem]{Proposition}
\newtheorem{lemma}[theorem]{Lemma}
\newtheorem{corollary}[theorem]{Corollary}

\crefname{theorem}{Theorem}{Theorems}
\crefname{definition}{Definition}{Definitions}
\crefname{proposition}{Proposition}{Propositions}
\crefname{lemma}{Lemma}{Lemmas}
\crefname{corollary}{Corollary}{Corollaries}

\makeatletter
\@addtoreset{equation}{section} 
\makeatother

\newcommand{\matr}[1]{\boldsymbol{\mathbf{#1}}}  
\newcommand{\Lapl}{\matr{L}}
\newcommand{\p}{\matr{p}}
\newcommand{\T}{\matr{T}}
\newcommand{\Id}{\matr{I}}

\newcommand{\dkl}{D_{\operatorname{KL}}}

\title{Evolution of Conditional Entropy for Diffusion Dynamics on Graphs}

\author{
 Samuel Koovely \\
  Department of Mathematical Modeling and Machine Learning\\
  University of Zurich\\
  Zurich, Switzerland \\
  \texttt{samuel.koovely@math.uzh.ch} \\
   \And
 Alexandre Bovet \\
  Department of Mathematical Modeling and Machine Learning\\
  University of Zurich\\
  Zurich, Switzerland\\
  \texttt{alexandre.bovet@math.uzh.ch}
}

\begin{document}
\maketitle
\begin{abstract}
  The modeling of diffusion processes on graphs is the basis for many network science and machine learning approaches. Entropic measures of network-based diffusion have recently been employed to investigate the reversibility of these processes and the diversity of the modeled systems. While results about their steady state are well-known, very few exact results about their finite-time evolution exist.
  Here, we introduce the conditional entropy of heat diffusion in graphs, and outline a mathematical framework that contextualizes diffusion and conditional entropy within the theories of continuous-time Markov chains and information theory. In particular,  we highlight that this entropic measure satisfies an information-theoretical version of the second law of thermodynamics, thereby providing a parallelism between diffusion dynamics on networks and their physical counterparts. Furthermore, we obtain explicit results for its evolution on complete, path, and circulant graphs, as well as a mean-field approximation for Erdős-Rényi graphs. We also obtain asymptotic results for general networks and provide bounds for the evolution of conditional entropy. Finally, we experimentally demonstrate several properties of conditional entropy for diffusion over random graphs, such as the Watts-Strogatz model.
\end{abstract}

% keywords can be removed

\section{Introduction}
Dynamics on graphs reflect and are governed by the structure over which they occur \cite{lambiotte_modularity_2021}. Various works in the field of complex networks have studied and used diffusion dynamics \cite{masuda_random_2017} on graphs; for instance Markov stability \cite{delvenne_stability_2010, schaub_markov_2012, lambiotte_random_2014} and InfoMap \cite{rosvall_maps_2008} are community detection methods based on diffusion on static graphs, while flow stability \cite{bovet_flow_2022} utilizes diffusion to cluster nodes in temporal networks.
Diffusion dynamics can be formally modelled as stochastic processes. In addition to their stochastic and statistical properties, one can study these processes from the perspective of information theory \cite{shannon_mathematical_1948}: the goal is to describe and characterize the evolution of a system through the quantification of exchange, or rates of exchange, of information, namely entropy and entropy rate, respectively. This point of view has seen a recent revival in interest for the study of complex networks, as information theory provides a unifying framework for describing various properties and phenomena of complex systems \cite{varley_information_2025}. There are different notions of entropy in networked systems; including entropy production \cite{schnakenberg_network_1976,nartallo-kaluarachchi_broken_2024}, which measure the divergence of a process from detailed balance, the Markov chain entropy \cite{gomez-gardenes_entropy_2008}, to be understood as an entropy rate rather than entropy despite its name, and the spectral entropy\cite{de_domenico_spectral_2016,ghavasieh_statistical_2020}, adapting the notion of Von Neumann entropy, originally introduced in the field of quantum information, to networks. The latter has been used to define a renormalization group for complex networks \cite{villegas_laplacian_2022, villegas_laplacian_2023} and for community detection \cite{villegas_multi-scale_2025}. 

Here, we introduce the conditional entropy of the heat diffusion process on graphs, which can also be modelled as a continuous-time Markov chain or a continuous-time random walk.
While thermodynamic laws based on information theory for general discrete-time Markov chains have already been established \cite{cover_elements_2006}, we demonstrate here that these results extend to the conditional entropy of continuous-time heat diffusion on graphs, making it a quantity that enables the study of diffusion dynamics on networks with physical intuition. 
We also provide the first exact results for its finite-time evolution on specific regular graphs, establishing heuristic upper and lower bounds for conditional entropy evolution curves. We then provide results for the asymptotic behaviour in general networks. Finally, we experimentally demonstrate several properties of conditional entropy for diffusion over random graphs, such as the Watts-Strogatz model, and derive an approximate solution for Erdős-Rényi graphs.

We begin with a section dedicated to the theoretical framework. In \cref{sec:TheoreticalFramework} we describe the notation for diffusion processes and information-theoretical quantities. This framework allows us to introduce the main object of interest: conditional entropy. In \cref{sec:ConditionalEntropy}, we highlight a general property for the local behavior of conditional entropy, analogous to the second law of thermodynamics. We therefore establish diffusion dynamics paired with conditional entropy as a candidate for a formalism of thermodynamics of networks. We then derive explicit spectral solutions that allow fast computation of conditional entropy evolution on some families of regular graphs (complete, path, and circulant graphs) based on the graph Fourier transform \cite{ortega_graph_2018} approach, and show connections with the discrete Fourier and cosine transforms. In the final part of the section, we derive general results about the asymptotic behaviour of conditional entropy and the relaxation time of diffusion. In \cref{sec:Simulations}, we extend our analysis to some random graphs (ER and Watts-Strogatz graphs) by simulations. Finally, we summarize some of our findings and contributions, and highlight some of the challenges and potential developments of this work.

\section{Theoretical Framework}\label{sec:TheoreticalFramework}
We now briefly introduce the notation of diffusion on graphs as a stochastic process. For readers unfamiliar with such notions, and those interested in a more complete discussion, we suggest reading \cite{masuda_random_2017}.

\subsection{Continuous-Time Diffusion on Graphs}
Let $ G = (V, E) $ be an undirected  graph, where:
\begin{itemize}
    \item $ V = \{v_1, v_2, \dots, v_N\} $ is the set of $ N  \in \mathbb{N}$ nodes, 
    \item $ E = \{e_1, e_2, \dots, e_M \} \subset V \times V $ is the set of $M \in \mathbb{N}$ edges connecting pairs of nodes.
\end{itemize}

We model continuous-time diffusion as a homogeneous Markov chain $\{ X(t) \}_{t \in \mathbb{R}^+}$ \cite{lanchier_stochastic_2017, masuda_random_2017} where the state at time $ t $ is $ X(t) \in  \left\{v_1, v_2, \dots, v_N \right\} $.  We use $p(t)$ to denote the law of $X(t)$, and $ p_i(t) $ represents the probability at node $ v_i $ at time $t$. We often manipulate $p(t)$ in vectorized form; in this case, we use the bold notation $\p(t)$ like for other vectors (lowercase) and matrices (uppercase). The Markov chain is characterized by an intensity matrix $\matr{Q}$ whose entries correspond to the instantaneous rate at which the chain transitions between states. The forward Kolmogorov equation then describes the evolution of the process:
\begin{equation}
    \frac{d \p^{\intercal}(t)}{dt} = \p^{\intercal}(t) \matr{Q},
\end{equation} whose solution is
\begin{equation}
    \p^{\intercal}(t) = \p^{\intercal}(0) e^{\matr{Q}t},
\end{equation}
and describe the distribution of the probability vector as a function of time. \\
More generally, for all times $t_1 \leqslant t_2$ we can introduce the concept of transition operator $\T(t_1, t_2)$ such that
\begin{equation}
     \forall i,j : T_{i,j}(t_1, t_2) \coloneqq p(X_j(t_2) | X_i(t_1)) = \frac{p(X_i(t_1), X_j(t_2))}{p(X_i (t_1))}.
\end{equation}
Given that $\{ X(t) \}$ is a Markov chain, we get
\begin{equation}
    \p^{\intercal}(t_2) = \p^{\intercal}(t_1) \T(t_1, t_2),
\end{equation}
and we have that
\begin{equation}
    \forall t_1 \leqslant t_2: \T(t_1, t_2) = e^{\matr{Q} (t_2-t_1)}.
\end{equation}
In this article, we focus on the heat diffusion process, which can also be seen as an edge-centric continuous-time random walk \cite{masuda_random_2017}. We provide results for continuous-time node-centric random walks in \cref{appendix:RW}. In this case of heat diffusion, the intensity matrix is based on the operator known as the combinatorial Laplacian (denoted $\Lapl$), defined as follows:
\begin{equation}
    \Lapl_{i,j} = \begin{cases} 
      d_i & \text{if } i=j, \\
      - 1 & \text{if } (v_i, v_j) \in E, i \neq j, \\
      0 & \text{otherwise},
   \end{cases}
\end{equation}
where $d_i$ indicates the degree of node $i$. \\
We can then write the corresponding Kolmogorov forward equation:
\begin{equation}
    \frac{d \p^{\intercal}(t)}{dt} = - \lambda \p^{\intercal}(t)  \Lapl, \text{ with the respective solution }  \p^{\intercal}(t) = \p^{\intercal}(0) e^{- \lambda \Lapl t}.
\end{equation}
where the parameter $\lambda \in \mathbb{R}^+$controls the rate of diffusion. On static networks, changing this parameter only corresponds to a rescaling of time; therefore, in the following, we will always consider it to be equal to 1.
\\
%Stationary Distribution
Irreducible continuous-time Markov chains are ergodic \cite{lanchier_stochastic_2017}. For heat diffusion this implies that on connected graphs, the probability vector converges to the stationary distribution $ \matr{\pi}^{\intercal} = \frac{1}{N} \matr{1}^{\intercal}_N = [\frac{1}{N}, \frac{1}{N}, \dots, \frac{1}{N}]$, where $\matr{1}_N$ is the $N$-th dimensional all-ones vector. \\
This stationary distribution $\pi$ reflects the equilibrium state of the diffusion process, where each node’s value remains constant over time, and satisfies the stationary condition:
\begin{equation} \label{eq:stationarity}
    \forall t \geqslant 0 : \matr{\pi}^{\intercal} = \matr{\pi}^{\intercal} e^{- \Lapl t}.
\end{equation}
For the remainder of the work, we reserve the symbols $X$ and $Y$ to indicate generic random variables and $\{X(t)\}, \{Y(t)\}$ for generic stochastic processes. Whenever we derive results specific to heat diffusion, we denote the stochastic process by $\{Z(t)\}$.

\subsection{Information Theory}

In this part, we introduce the information-theoretical backbone of our work. We follow loosely the notation from \cite{cover_elements_2006}. A more formal introduction to information theory can be found in \cite{csiszar_information_2011}, and for an exposition highlighting aspects relevant in complex systems, we recommend \cite{varley_information_2025}.

\begin{definition}
    The (Shannon) entropy of a discrete random variable $X$ taking values on the state space $\mathcal{X}$ with elements $x$ is defined as
    \begin{equation}
        H(X) \coloneqq - \sum_{x \in \mathcal{X}} p(x) \log(p(x)).
    \end{equation}
    Since random variables that are identically distributed have the same entropy, we will sometimes use the expression $H(p)$ meaning $H(X)$, where $X$ is $p$ distributed.
\end{definition}
Conceptually, the choice of the logarithm basis is irrelevant in our analysis. We opted for the natural basis.
\begin{definition}
    For $(X,Y)$ taking values in the state space $\mathcal{X} \times \mathcal{Y}$ and  with distribution $p(x,y)$, we define the conditional entropy $H(Y | X)$ as:
    \begin{equation}
        H(Y | X) \coloneqq \sum_{x \in \mathcal{X}} p(x) \sum_{y \in \mathcal{Y}} p(y | x) \log(p(y | x)).
    \end{equation}
\end{definition}
\begin{definition}
    The Kullback-Leibler divergence (in the following KL divergence) between probability mass functions $p$ and $q$ defined on the same state space and sharing the same support is defined as:
    \begin{equation}
        \dkl(p || q) \coloneqq  \sum_x p(x) \log \left( \frac{p(x)}{q(x)} \right).
    \end{equation}
\end{definition}
The KL divergence  $\dkl(p || q)$  provides a notion of distance between $p$ and $q$. However, it is not symmetrical, so it is more correct to consider it as a measure of inefficiency in approximating $p$ with $q$. 
We can use it to measure distances between conditional probability distributions:
\begin{definition}
    The KL divergence between conditional probability functions $p(y|x)$ and $q(y|x)$ is defined as
    \begin{equation}
        \dkl(p(y | x) || q(y | x)) \coloneqq \sum_x p(x) \sum_y p(y | x) \log \left( \frac{p(y | x)}{q(y | x)}\right).
    \end{equation}
\end{definition}
The KL divergence has the following two properties that we will use later. We omit their proofs, but one can find them in \cite[Thm. 2.5.3]{cover_elements_2006}, and \cite[Lemma 3.11]{csiszar_information_2011}.

\begin{proposition}[Chain Rule for KL Divergence] \label{prop:chainKL}
    \begin{equation}
        \dkl(p(x,y) || q(x,y)) = \dkl(p(x) || q(x)) + \dkl(p(y | x) || q(y | x)).
    \end{equation}
\end{proposition}

\begin{proposition}[Pinsker Inequality] \label{prop:Pinsker}
    Let $p$ and $q$ be distributions defined on the same probability space. Then
     \begin{equation}
         \dkl(p\| q) \geqslant \frac{1}{2} \|p- q \|_1^2.
     \end{equation}
     We notice that for this result, logarithm bases other than Euler's number $e$ lead to slightly different versions of the inequality, which differ by a constant factor from our formulation.
\end{proposition}

As we will often manipulate distributions in their vectorized form, we will often write for convenience expressions such as $H(\p)$ meaning $H(p)$, and $\dkl(\p || \matr{q})$ meaning $\dkl(\p^{\intercal} || \matr{q}^{\intercal})$ or $\dkl(p || q)$.

\section{Conditional Entropy of Diffusion Dynamics on Graphs} \label{sec:ConditionalEntropy}
We now have all the ingredients to introduce the object of interest of this article: the conditional entropy curve for diffusion dynamics on graphs.
\begin{definition}
    Consider a Markov chain $\{ X(t) \}_{t \in \mathbb{R}^+}$ describing a diffusion dynamic. We denote the $i$-th row entropy associated with the probability distribution of transitioning from node $i$ at time $0$ to any other node at time $t$ by
    \begin{equation}
        H_i (t|0) \coloneqq - \sum_j T_{i,j}(0,t) \log(T_{i,j}(0,t)).
    \end{equation}
    
    Then, we define the conditional entropy of the chain as
    \begin{equation} \label{eq:conditionalH_rows}
        H(X(t) | X(0)) \coloneqq \sum_{i} p_i (0) H_i(t | 0).
    \end{equation} 
\end{definition}
 This object allows us to describe a formalism for thermodynamics on networks. The heat is modeled by $\{Z(t)\}$, whereas the conditional entropy $H(Z(t) |Z(0))$ plays the role of the entropy. We depict in \cref{fig:trajectories_path} the trajectories of $Z(t)$ in a path graph starting from the initial condition $Z(0) = \delta_1$ (i.e., initially all the probability is concentrated at one end of the path graph). The entries of $Z(t)$ converge to the stationary distribution values (see \cref{eq:stationarity}), not necessarily monotonically. We also depict the entropy curve of this diffusion dynamic, which shows a typical monotonic convergence towards an asymptotic value (discussed later in \cref{subsec:asymptotic_behaviour}).
\begin{figure}
    \centering
      \label{fig:trajectories_path}\includegraphics[scale=0.4]{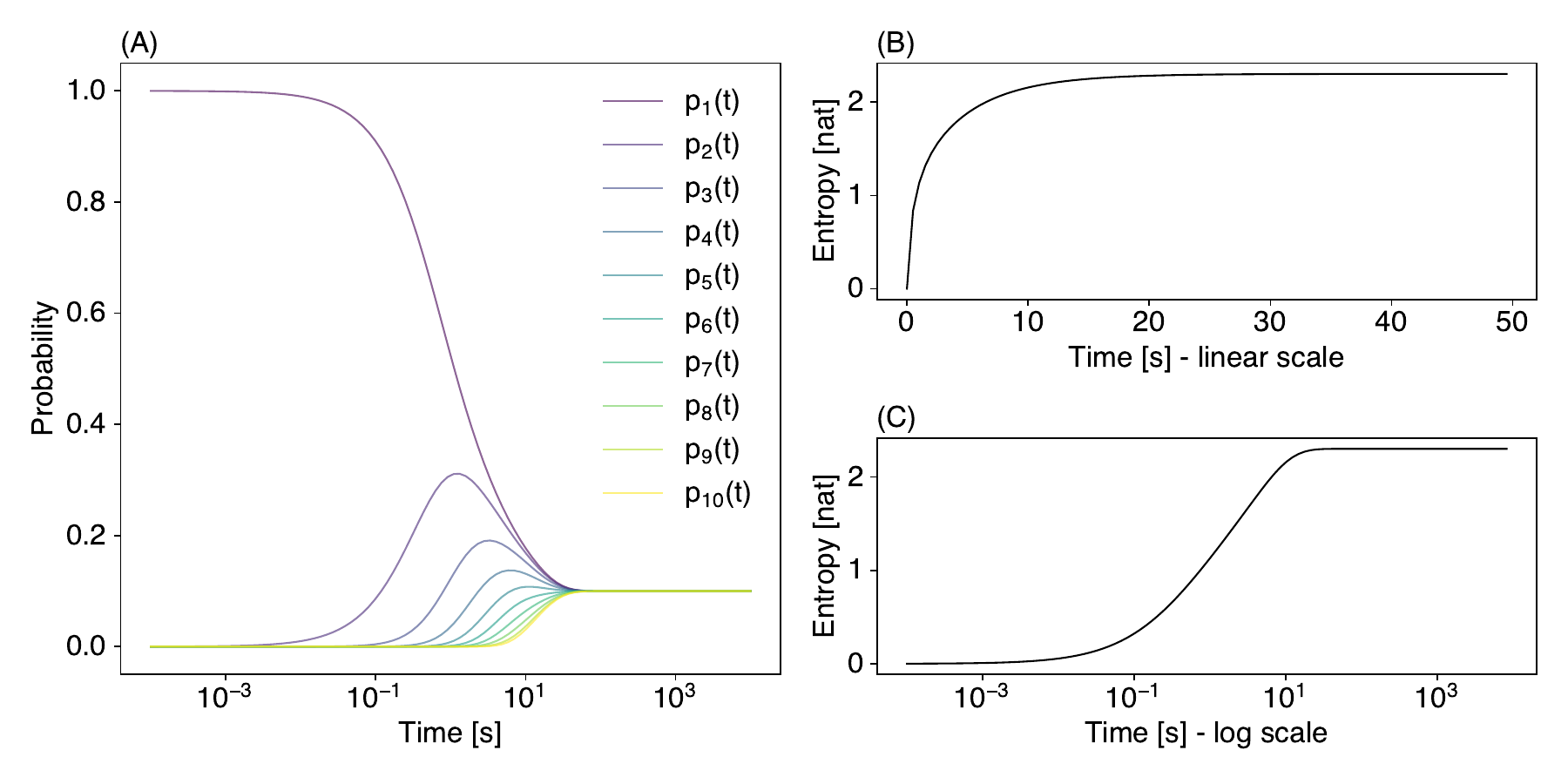}
      \caption{Heat diffusion on a path graph with 10 nodes with initial condition $\p(0) = \delta_1$. (A) Initially, all the transition probability is concentrated on the source node itself; over time, the heat diffuses. At stationarity, transition to any node is equally likely, but convergence to the limit value is not necessarily monotonic. (B) Conditional entropy of the process conditioned on the $\delta_1$ initial condition. The line shows the evolution of the conditional entropy on a linear time scale. The entropy starts at a value 0 and grows monotonically towards an asymptotic value, which is reached at stationarity. (C) The line shows the same entropy curve but on a logarithmic time scale and displays a typical logistic-like growth.}
\end{figure}

\subsection{Thermodynamics of Graph-Diffusion} \label{subsec:Thermodynamics}
The first law of thermodynamics asserts that the total energy is conserved in thermodynamic processes. Similarly, we can say that Markov chains on finite state spaces are also conservative, as $\sum_{i=1}^N p_i(t) = \sum_{i=1}^N p_i(0) = 1$ holds trivially for every $t$. 
The choice of the correct entropic quantity is key to deriving a thermodynamic formalism for networks that respects an equivalent of the second law of thermodynamics and is therefore coherent with our intuition. For instance, directly measuring the entropy of the state vector would not necessarily guarantee this property \cite{cover_elements_2006, halliwell_physical_1996}. On the other hand, choosing to measure directly the KL divergence of the process from the stationary distribution would always be monotonically decreasing \cite{cover_elements_2006}; however, for initial conditions close to stationarity, this would lead to a dull system, uninformative of the underlying graph structure, and therefore practically of little interest.\\
Here, we suggest the use of conditional entropy as an entropic measure that satisfies the information theoretical version of the second law of thermodynamics. The conditional entropy at time $0$ is trivially equal to zero, regardless of the initial distribution on which we are conditioning, and \cite[Ch. 4.4]{cover_elements_2006} outlines a proof of the monotonic growth for the conditional entropy of discrete-time Markov chains when the process starts from the stationary distribution. In that case for a Markov chain $\{ X(t) \}_{t \in \mathbb{R}^+}$ we have:
      \begin{equation}
          X(0) \text{ stationary}, 0 \leqslant t_1 \leqslant t_2 \implies H(X(t_2) | X(0)) \geqslant H(X(t_1) | X(0)).
      \end{equation}
We provide an original argument based on the connection between conditional entropy and the KL divergence. Let us denote by $p(0,t)$ the joint distribution of the dynamic at times $0$ and $t$, and by $p(0) \otimes p(t)$ the product distribution of the laws of $X(0)$ and $X(t)$. Then the chain rule of Shannon entropy and the definition of mutual information \cite{cover_elements_2006} imply that
\begin{equation} \label{eq:KL_conditionalH}
    H(X(t) | X(0)) = H(X(t)) - \dkl(p(0,t) \| p(0) \otimes p(t)).
\end{equation}
We can therefore interpret the conditional entropy of a diffusion process as the interplay between the entropy of the variable $X(t)$ converging to the stationary distribution and its decreasing dependence on the initial condition $X(0)$.

Since the process is stationary, the first term is constant. The data-processing inequality for KL divergence \cite[Lemma 3.11]{csiszar_information_2011} implies that the second term is non-increasing, hence \cref{eq:KL_conditionalH} holds.

Additionally, in the case of heat diffusion, monotonicity holds independently of the initial condition. We could prove it with \cref{eq:KL_conditionalH} by combining again based on the data-processing inequality with the fact that $H(Z(t))$ is a non-decreasing function independently from $p(0)$. However, we present an alternative argument, closer to the proof given in \cite{cover_elements_2006}. Our approach highlights that convergence happens row-wise, which is conceptually interesting on its own and might be potentially useful for other proofs as well, and holds under weaker assumptions: it can be a more generic (potentially inhomogeneous) Markov chain, as long as its stationary distribution is a uniform one. The key result used in the proof is again a version of the data-processing inequality for KL divergence; which we put in \cref{app:DPI} for completeness.

\begin{theorem} [Second Law of Thermodynamics] \label{Thm:2ndlaw}
    Let $\{ X(t) \}_{t \in \mathbb{R}^+}$be a Markov chain, and let the uniform distribution $\matr{\pi}$ be its stationary distribution. Then we have:
      \begin{equation}
          \forall X(0), t_1 \leqslant t_2: H(X(t_2) | X(0)) \geqslant H(X(t_1) | X(0)).
      \end{equation}
\end{theorem} 

\begin{proof}
We use vector-matrix bold notation and set $\tau \coloneqq t_2-t_1 \geq 0$ and define
$\p^{(i)}(t_1) \coloneqq \T_{i,:}(0,t_1)$,
that is, the $i$-th row of the transition matrix from time $0$ to time $t_1$.
Then, by the semigroup property of the Markov chain,
$ \p^{(i)}(t_2) = \T_{i,:}(0,t_1)\T(t_1,t_2) = \T_{i,:}(0,t_2)$.
% Since the underlying graph is connected, the uniform distribution is the only stationary one, independently of the initial condition.
Applying the data-processing inequality from \cref{lem:DPI_markov_kernel} with $ \p(t_1)=\p^{(i)}(t_1)$ and $\matr{q}(t_1)=\boldsymbol{\pi}$, we obtain
\begin{equation}
\label{eq:KL_row_decrease}
    \dkl\!\left(
        \p^{(i)}(t_2)\,\|\,\boldsymbol{\pi}
    \right)
    \leq
    \dkl\!\left(
        \p^{(i)}(t_1)\,\|\,\boldsymbol{\pi}
    \right),
\end{equation}
or equivalently,
\begin{equation}
\label{eq:KL_row_decrease_transition}
    \dkl\!\left(
        \T_{i,:}(0,t_2)\,\|\,\boldsymbol{\pi}
    \right)
    \leq
    \dkl\!\left(
        \T_{i,:}(0,t_1)\,\|\,\boldsymbol{\pi}
    \right).
\end{equation}

We now relate this KL divergence to the entropy of the $i$-th row of the transition matrix.
Since $\pi_j = 1/N$ for every $j$, we have
\begin{equation*}
    \dkl\!\left(
        \T_{i,:}(0,t)\,\|\,\boldsymbol{\pi}
    \right)
    =
    - H_i(t\mid 0) -\sum_{j=1}^N T_{ij}(0,t)
    \log(\pi_j) =
    - H_i(t\mid 0) + \log (N);
\end{equation*}
thus
\begin{equation}
\label{eq:entropy_KL_identity}
    H_i(t\mid 0)
    =
    \log (N)
    -
    \dkl\!\left(
        \T_{i,:}(0,t)\,\|\,\boldsymbol{\pi}
    \right).
\end{equation}
Combining \cref{eq:KL_row_decrease_transition,eq:entropy_KL_identity} gives
\begin{equation}
    H_i(t_2\mid 0)
    \geq
    H_i(t_1\mid 0)
\end{equation}
for every $i$. Finally, by definition, $H(X(t)\mid X(0))
    =
    \sum_i p_i(0) H_i(t\mid 0).$
Averaging the inequalities $H_i(t_2\mid 0) \geq H_i(t_1\mid 0)$
with respect to the initial distribution $p(0)$ gives $H(X(t_2)\mid X(0)) \geq H(X(t_1)\mid X(0))$.
Therefore, $H(X(t)\mid X(0))$ is non-decreasing in time.
\end{proof}

We conclude this section by presenting a counterexample to the possibility of extending \cref{Thm:2ndlaw} to general chains, not only heat diffusion. Relation \cref{eq:KL_conditionalH} tells us where to look for such a counterexample.
\begin{proposition}
    Consider the Markov chain with $N$ states determined by the transition matrix
    $$
    \T(0,t) \coloneqq\begin{bmatrix} 1 & 0 & \dots & 0 \\ 1-e^{-t} & e^{-t} &  & \vdots \\ \vdots &  & \ddots & 0 \\ 1-e^{-t} & 0 & \dots & e^{-t} \end{bmatrix},
    $$ and the initial condition $\p(0) \coloneqq [0, \frac{1}{N-1}, \dots, \frac{1}{N-1}]$. Then $H(X(t) |X(0))$ is not a monotonically non-decreasing function.
\end{proposition}
\begin{proof}
    One can easily compute that
    \begin{equation}
        H(X(t)) = (e^{-t}-1) \log(1-e^{-t}) + (\log(N-1) + t) e^{-t},
    \end{equation}
    and
    \begin{equation}
        \dkl(p(0,t) \| p(0) \otimes p(t)) = e^{-t} \log(N-1).
    \end{equation}
    Then, by \cref{eq:KL_conditionalH} we get \begin{equation}
        H(X(t)|X(0)) = (e^{-t}-1) \log(1-e^{-t}) +  te^{-t},
    \end{equation} which is not identical to 0. $\lim_{t \to \infty}  H(X(t)|X(0)) =  H(X(0)|X(0)) = 0$, so the function must decrease at a certain point.
\end{proof}

\subsection{Exact Results of the Evolution of Conditional Entropy on some Graphs} \label{subsec:extact_results}
The evolution of the conditional entropy depends on the graph's structure and is encoded in the spectral components of its Laplacian matrix. For certain classes of graphs with certain levels of regularity, the spectrum and eigenspaces are known \cite{van_mieghem_graph_2023}. Here, we derive them again and use them to obtain the eigendecomposition of the transition matrices for heat diffusion on these classes of graphs and obtain exact formulas for the evolution of conditional entropy.
Knowing these decompositions is, in turn, useful for fast computation of entropy curves and for understanding asymptotic behaviours. \\
Given a graph $G$ of size $N$, we always consider the spectrum $\{ 0 = \lambda_1, \dots, \lambda_N \}$ of its Laplacian (or in fewer cases adjacency) matrix to be in ascending order. Additionally, with a slight abuse of notation, we write $\lambda_i (G)$ to indicate the $i$-th eigenvalue of $G$'s Laplacian.

\subsubsection{Complete Graph}
In the case of a complete graph, where each node is connected to every other node in the network, given its extreme regularity, we can explicitly write its transition probability matrix. We first need to compute its Laplacian's eigendecomposition.
\begin{lemma} \label{lemma:spectrumKN}
    Let $\Lapl$ be the Laplacian of the complete graph $K_N$ with $N$ nodes. Then:
    \begin{equation}
        \lambda_1(K_N) = 0, \lambda_2 (K_N) = \dots = \lambda_N(K_N) = N,
    \end{equation} and
    \begin{equation}
        \Lapl = \matr{U} \matr{\Lambda} \matr{U}^{\intercal},
    \end{equation}
    where $\matr{\Lambda} = \mathrm{diag}(0, N, \dots, N) $ and $ U $ is an orthogonal matrix whose first column is $ \frac{1}{\sqrt{N}} \matr{1}_N $, and whose remaining columns form an orthonormal basis for the subspace orthogonal to $ \matr{1}_N $.
\end{lemma}
\begin{proof}
    Consider the all-ones vector $\matr{1}_N \in \mathbb{R}^N$ and the all-ones matrix $\matr{J}_N = \matr{1}_N \matr{1}_N^{\intercal}$.
    Notice how  $\matr{1}_N$ belongs to the right kernel of the Laplacian of each graph (and in particular $K_N$) and that 
    $ \Lapl = N \Id_N - \matr{J}_N $, where $\Id_N$ is the $N$-dimensional identity matrix. The rank of $\matr{J}$ is equal to 1, so we can conclude that, except for 0, every other eigenvalue of $\Lapl$ is equal to $N$ (which therefore has multiplicity $N-1$), a fact established already in \cite{anderson_eigenvalues_1971}. \\
    The spectral decomposition of $ \Lapl $ is then $ \Lapl = \matr{U} \matr{\Lambda} \matr{U}^{\intercal},$ where $ \matr{\Lambda} = \mathrm{diag}(0, N, \dots, N) $ and $ \matr{U} $ is an orthogonal matrix whose first column is $ \frac{1}{\sqrt{N}} \matr{1}_N $, and whose remaining columns form an orthonormal basis for the subspace orthogonal to $ \matr{1}_N $.
\end{proof}
We can now get the formula of the transition matrix.
\begin{proposition}
    Let $\{Z(t)\}$ be the heat diffusion dynamic on the complete graph $K_N$ with $N$ nodes. Then its transition matrices are given by:
    \newcommand{\theatone}{\frac{1}{N} + \frac{N-1}{N} e^{- N t}}
    \newcommand{\theattwo}{\frac{1}{N} - \frac{1}{N}e^{- N t}}
    
    \begin{equation} \label{eq:CTHeat_complete}
        T_{i,j}(t) = \begin{cases}
        \theatone & \text{if } i = j,\\
       \theattwo & \text{if } i \neq j.
    \end{cases}
    \end{equation}
\end{proposition}

\begin{proof}    
    Based on the preceding lemma, the matrix exponential form is
    \begin{equation*}
        e^{-t\Lapl} = \matr{U} e^{-t \matr{\Lambda}} \matr{U}^{\intercal} = \matr{U} \, \mathrm{diag}(1, e^{-tN}, \dots, e^{-tN}) \, \matr{U}^{\intercal}.
    \end{equation*}
    We now define the projection matrix onto the span of $ \matr{1}_N $ as $\matr{P} = \frac{1}{N} \matr{J}_N$,
    and observe that $ \Id_N = \matr{P} + (\Id_N - \matr{P}) $. Since $ \matr{P} $ projects onto the eigenspace of $ \lambda = 0 $, and $ \Id_N - \matr{P} $ projects onto the eigenspace of $ \lambda = N $, we obtain
    \begin{equation*}
        e^{-t \Lapl} = \matr{P} + e^{-tN} (\Id_N - \matr{P}).
    \end{equation*}
    Expanding this expression yields the final result:
    \begin{equation}
        e^{-t\Lapl} = e^{-tN} \Id_N + \left(1 - e^{-tN}\right) \frac{1}{N} \matr{J}_N,
        \label{eq:T_complete}
    \end{equation}
    whose entries are specified in \cref{eq:CTHeat_complete}.
\end{proof}
We now state the explicit formula of the conditional entropy of the heat process. The computation of the formula can be found in \cref{appendix:approx_complete}, and the analogous formulas for transition matrices and conditional entropy of random walk diffusion are in \cref{eq:exactRWcomplete}.

\begin{corollary}
    The conditional entropy $H(Z (t) | Z(0))$ on the complete graph is independent of the choice of the initial condition $\p(0)$ and is given by
        \begin{equation} \label{eq:exactHcomplete}
            \begin{aligned}
                 H(Z(t) | Z(0)) &= \log(N) 
                - \frac{N-1}{N} e^{-  N t} 
                \log\left(\frac{1 + (N-1) e^{- N t}}{1 - e^{- N t}}\right) \\
                &- \frac{N-1}{N} \log\left(1 - e^{- N t}\right) 
                - \frac{1}{N} \log\left(1 + (N-1) e^{- N t}\right)
                \end{aligned}
        \end{equation}
\end{corollary}

\subsubsection{Path Graphs}
%--- Laplacian ------------------------------------------------------
For the path graph $P_N$ with vertices $V=\{1,\dots,N\}$ and edges
$(i,i+1)$ for $1\le i\le N-1$ the Laplacian matrix $\Lapl$ is
\begin{equation}
    \Lapl_{i,j}= \begin{cases}
1 & \text{ if } i=j\in\{1,N\},\\
2 & \text{ if } 1<i<N,\; i=j,\\
-1 & \text{ if } |i-j|=1,\\
0 & \text{otherwise}.
\end{cases}
\end{equation}
We now state its eigendecomposition.
\begin{proposition}
    Let $\Lapl$ be the Laplacian of the path graph $P_N$ with $N$ nodes. Then its spectrum is given by the set of eigenpairs $\{ \lambda_k, v^{(k)}\}_{1 \leqslant k \leqslant N}$ where
    \begin{equation} \label{eq:spectrum_path}
        \lambda_k = 2 \left( 1-\cos \left( \frac{\pi(k-1)}{N}\right) \right),
    \end{equation} and
    \begin{equation} \label{eq:evectors_path}
        \forall 1 \leqslant j \leqslant N: v_j^{(k)} = \sqrt{\frac{2-\delta_{k1}}{N}}\, \cos \left( \frac{\pi(k-1) \left( j-\frac{1}{2} \right) }{N} \right).
    \end{equation}
\end{proposition}
\begin{proof}
    We look for non-zero $\matr{v}=(v_1,\dots,v_N)^{\intercal}$ and $\lambda$ such that
    $\Lapl \matr{v}=\lambda\matr{v}$. For interior indices $j\in\{2,\dots,N-1\}$ we have
    \begin{equation}
    -\,v_{j-1}+2v_j-v_{j+1}= \lambda v_j,
    \end{equation}
    while the endpoints satisfy
    \begin{equation} \label{eq:endpoint_path}
    v_1-v_2=\lambda v_1,\qquad v_N-v_{N-1}=\lambda v_N.
    \end{equation}
    Set $v_j=\cos \left((j-\frac{1}{2})\theta\right)$. Using
    $\cos(x-\theta)+\cos(x+\theta)=2\cos (x) \cos(\theta)$ with $x=(j-\frac{1}{2})\theta$,
    the interior equation gives
    \begin{equation*}
    -\,\cos(x-\theta)+2\cos(x) - \cos(x+\theta)
    =2(1-\cos(\theta))\cos(x);
    \end{equation*}
    hence
    \begin{equation}\label{eq:lambda-theta}
    \lambda=2\left(1-\cos(\theta) \right).
    \end{equation}
    With $v_1=\cos(\frac{\theta}{2})$ and $v_2=\cos(\frac{3\theta}{2})$, the left endpoint equation of \cref{eq:endpoint_path} becomes
    \begin{equation*}
    v_1-v_2
    =\cos \left(\frac{\theta}{2}\right)-\cos \left( \frac{3\theta}{2} \right)
    =2\sin \left( \theta \right) \sin \left(\frac{\theta}{2}\right)
    =4\sin^2 \left(\frac{\theta}{2}\right)\cos \left(\frac{\theta}{2}\right)
    =2(1-\cos \left(\theta \right) ) v_1.
    \end{equation*}
    The left boundary condition is therefore automatically satisfied once \cref{eq:lambda-theta} holds.
    \\
    Quantization of the spectrum is achieved by imposing the right endpoint equation of \cref{eq:endpoint_path}. Let $\alpha \coloneqq(N-\frac{1}{2})\theta$. Then
    $v_N=\cos(\alpha)$ and $v_{N-1}=\cos(\alpha-\theta)$, and the boundary condition reads
    \begin{equation}
    \cos (\alpha)-\cos(\alpha-\theta)=2(1-\cos(\theta)) \cos(\alpha).
    \end{equation}
    Using $\cos(\alpha-\theta)=\cos (\alpha)\cos(\theta)+\sin (\alpha)\sin(\theta)$ and rearranging,
    \begin{align*}
    0 & =(1-\cos(\theta))\cos (\alpha)+\sin (\alpha)\sin(\theta)
    =\cos (\alpha)-\left(\cos(\theta)\cos (\alpha)-\sin(\theta)\sin (\alpha)\right) \\
     &= \cos (\alpha)-\cos(\alpha+\theta).
    \end{align*}
    Hence,
    \begin{equation}
    \cos \left( \left(N+\frac{1}{2} \right)\theta\right)=\cos \left( \left(N-\frac{1}{2} \right) \theta\right).
    \end{equation}
    Therefore either $(N+\frac{1}{2})\theta=(N-\frac{1}{2})\theta+2\pi m$, which on $[0,\pi]$
    gives only $\theta=0$ (the zero eigenvalue), or
    $(N+\frac{1}{2})\theta=-(N-\frac{1}{2})\theta+2\pi m$, i.e.
    \begin{equation}
    2N \theta=2\pi m \Longrightarrow  \theta=\frac{\pi m}{N}.
    \end{equation}
    The choice $m=N$ gives the zero vector, so we take $m=0,1,\dots,N-1$ and index them as
    $k=1,\dots,N$ with
    \begin{equation}
    \theta_k=\frac{\pi(k-1)}{N}.
    \end{equation}
    
    Substituting $\theta_k$ into \eqref{eq:lambda-theta} yields
    \begin{equation}
    \lambda_k=2\left(1-\cos\frac{\pi(k-1)}{N}\right),
    \end{equation}
    and the corresponding eigenvector components are
    \begin{equation}
    v^{(k)}_j=\cos\left(\frac{\pi(k-1)}{N}\left(j-\frac{1}{2}\right)\right),
    j=1,\dots, N,
    \end{equation}
    which may be normalized to match \eqref{eq:evectors_path}.
\end{proof}
The vectors $\{\matr{v}^{(k)}\}_{k=1}^N$ are orthonormal and constitute the (type-II) discrete cosine transform (DCT-II) basis \cite{ahmed_discrete_1974}.
\\
Let $\matr{U}=[\matr{v}^{(1)}\;|\;\cdots\;|\;\matr{v}^{(N)}]$ and
$\matr{D}(t)=\operatorname{diag}  \left( e^{-\lambda_1 t}, \dots, e^{-\lambda_N t} \right) $.  
The heat-diffusion transition matrix is then:
\begin{equation}
    \T(t)=e^{-t \Lapl}=\matr{U} \matr{D}(t) \matr{U}^{\intercal}, \text{ with }
    T_{i,j}(t)=\sum_{k=1}^{N}e^{-\lambda_k t}\,v_i^{(k)}\,v_j^{(k)}.
\end{equation}
The complete and path graphs play an important role in the understanding of diffusion, as they represent extremal cases in terms of connectedness. In particular, in \cref{subsec:asymptotic_behaviour} we show that the speed of convergence to stationarity is in some sense maximal for the former and minimal for the latter. We therefore use their entropy curves in our simulation results to bound a region of space where we expect to see entropy curves of heat diffusion on graphs with the same number of nodes.
\subsubsection{Circulant Graphs}
Another case where we can exploit the high regularity of the structure to obtain an explicit solution is that of circulant graphs.
A circulant graph with $N$ nodes has vertex set $V = \{0,1,\dots,N-1\}$ and its structure is then entirely defined by its step-set $S \subseteq \{1,2,\dots,\lfloor N/2\rfloor\}$ indicating at which (regular) intervals there are edges:
\begin{equation*}
    E \;=\;
\left\{\ \{i,j\}\subseteq V \;\bigm|\; j-i \equiv s \pmod{N}\text{ for some }s\in S
\right\}.
\end{equation*}
We write $C_N(S) = (V,E) $ to indicate an undirected circulant graph. We then have that every node $i \in V$ has the same degree:
\begin{equation*}
d =
\begin{cases}
2|S| & \text{ if } N\text{ odd or } N/2\notin S,\\
2|S| -1  & \text{ if } N \text{ even and } N/2\in S.
\end{cases}
\end{equation*}
Now, let $C_N(S)$ be the undirected circulant graph with step–set $S$. Its adjacency matrix is then a circulant matrix \cite{davis_circulant_1994}.  
\begin{equation*}
    \matr{A}= \operatorname{circ}(c_0,\dots,c_{N-1}) \coloneqq
    \begin{bmatrix}
    c_0 & c_1 & \cdots & c_{N-1} \\
    c_{N-1} & c_0 & \cdots & c_{N-2} \\
    \vdots & \vdots & \ddots & \vdots \\
    c_1 & c_2 & \cdots & c_0 \\
\end{bmatrix}
\text{, where } 
     c_j= \begin{cases}
        1 & \text{ if } j\in S\ \text{or}\ N-j\in S, \\
        0 & \text{otherwise}.
    \end{cases}.
\end{equation*}
We then have the following eigendecomposition of its Laplacian.
\begin{proposition}
    The (unordered) spectrum of the Laplacian of $C_N(S)$ is 
    \begin{equation}
    \left\{ \lambda_k = d - 2\sum_{s\in S}\cos \left( \frac{2\pi s (k-1)}{N} \right)\right\}_{1 \leqslant k \leqslant N}.
    \end{equation}
    
\end{proposition}
\begin{proof}
    Consider the $N$-th dimensional shift matrix $\operatorname{circ}(0,1,0, \dots, 0)$. It is easy to verify that its spectrum is  $\{ \omega_N^{-k} \}_{0 \leqslant k < N}$, where $\omega_N \coloneqq e^{-2\pi i/N}.$
    Consider the discrete Fourier matrix  
    % Fourier diagonalisation ---------------------------------
    \begin{equation} \label{eq:DFTmatrix}
    \matr{F}=\frac1{\sqrt N}\left[\omega_N^{\,jk}\right]_{0\le j,k<N};
    \end{equation}
    then we have a diagonal decomposition of the shift matrix: 
    $\matr{F}^{*} \operatorname{diag}(\omega_N^0,\dots,\omega_N^{-(N-1)}) \matr{F}$. Since $\matr{A} = \sum_{0 \leqslant j \leqslant N-1} c_j \operatorname{circ}(0,1,0, \dots, 0)^j$, $\matr{A}$ shares the same eigenspaces with $\operatorname{circ}(0,1,0, \dots, 0)$ and if $x$ is a eigenvalue of the shift matrix, $\sum_{0 \leqslant j \leqslant N-1} c_j x^j
    $ is an eigenvalue of $\matr{A}$.
    Then we have a diagonal decomposition of the adjacency matrix:
    $\matr{A} = \matr{F}^{*} \operatorname{diag}(\mu_0,\dots,\mu_{N-1}) \matr{F}$, with (non-ordered) eigenvalues
    \begin{equation}
        \mu_k =\sum_{j=0}^{N-1}c_j \omega_N^{-jk}
        =2\sum_{s\in S}\cos \left( \frac{2\pi sk}{N} \right) , k \in \{0,\dots,N-1 \}.
    \end{equation} 
    % Laplacian  &  Heat (diffusion) operator-----------------------------------------------
    Let $d$ be the degree of every node. Then from $\Lapl=d \Id_N-\matr{A}.$ we get that the spectrum of the Laplacian of $C_N(S)$ is $\{ \lambda_k = d - 2\sum_{s\in S}\cos \left( \frac{2\pi s (k-1)}{N} \right)\}_{1 \leqslant k \leqslant N}$.
\end{proof}

\begin{proposition}
    The transition Matrix of heat diffusion on a circulant graph has entries
    \begin{equation}
        T_{i,j}(t) = \frac{1}{N}\sum_{k=0}^{N-1} \exp  \left( -t \lambda_{k+1} \right) \omega_N^{\,k ((j-i)\bmod n)}.
    \end{equation}
\end{proposition}
\begin{proof}
    \begin{equation*}
        e^{-t\Lapl}
         =e^{-t(d \Id_N-\matr{A})}
         =e^{-td \Id_N}\,e^{t \matr{A}}
         =e^{-dt} e^{t\matr{A}},
    \end{equation*}
    or in its spectral form $e^{-t \Lapl} = \matr{F}^{*}
            \operatorname{diag}  \left( e^{-t(d-\mu_0)},\dots, e^{-t(d-\mu_{N-1})} \right) 
            \matr{F}.$
    % 3.  Explicit entries (circulant kernel) --------------------
    From this, we get the entries of the first row of the transition matrix
    \begin{equation*}
        \forall r \in \{0,\dots,n-1\}: h_t(r) \coloneqq \left( e^{-t\Lapl} \right) _{0,r} =
               \frac{1}{N}\sum_{k=0}^{N-1} \exp \left( -t \lambda_{k+1} \right) \omega_N^{\,k r}.
    \end{equation*}
    \\
    Since $e^{-t\Lapl}$ is circulant, every entry depends only on the distance
    $r=j-i\bmod n$:
    \begin{equation*}
        \left( e^{-t\Lapl} \right)_{i,j}=h_t(j-i\bmod n).
    \end{equation*}
\end{proof}
In the special case of the undirected cycle, where $S=\{1\}$ and nodes have degree $d=2$, one gets
\begin{equation}
    \mu_k = 2\cos \left( \frac{2\pi k}{N} \right) , \text{ and }
    h_t(r)=e^{-2t} \frac{1}{N}\sum_{k=0}^{N-1} \exp \left( 2t\cos \left( \frac{2\pi k}{N} \right) \right) \omega_N^{k r}.
\end{equation}
In \cref{fig:circulant_curves} we show how edge density and diameter of a circulant graph influence the evolution of heat conditional entropy on top of it. The density of a circulant graph is equal to $\frac{d}{N-1} \approx \frac{2 |S|}{N-1}$. The graph is connected if and only if the step-set $S$ generates $\mathbb{Z}_N$; hence, finding its diameter is equivalent to the Cayley graph diameter problem over $\mathbb{Z}_N$. \\

\begin{figure}[htbp]
  \centering
  \includegraphics[scale=0.40]{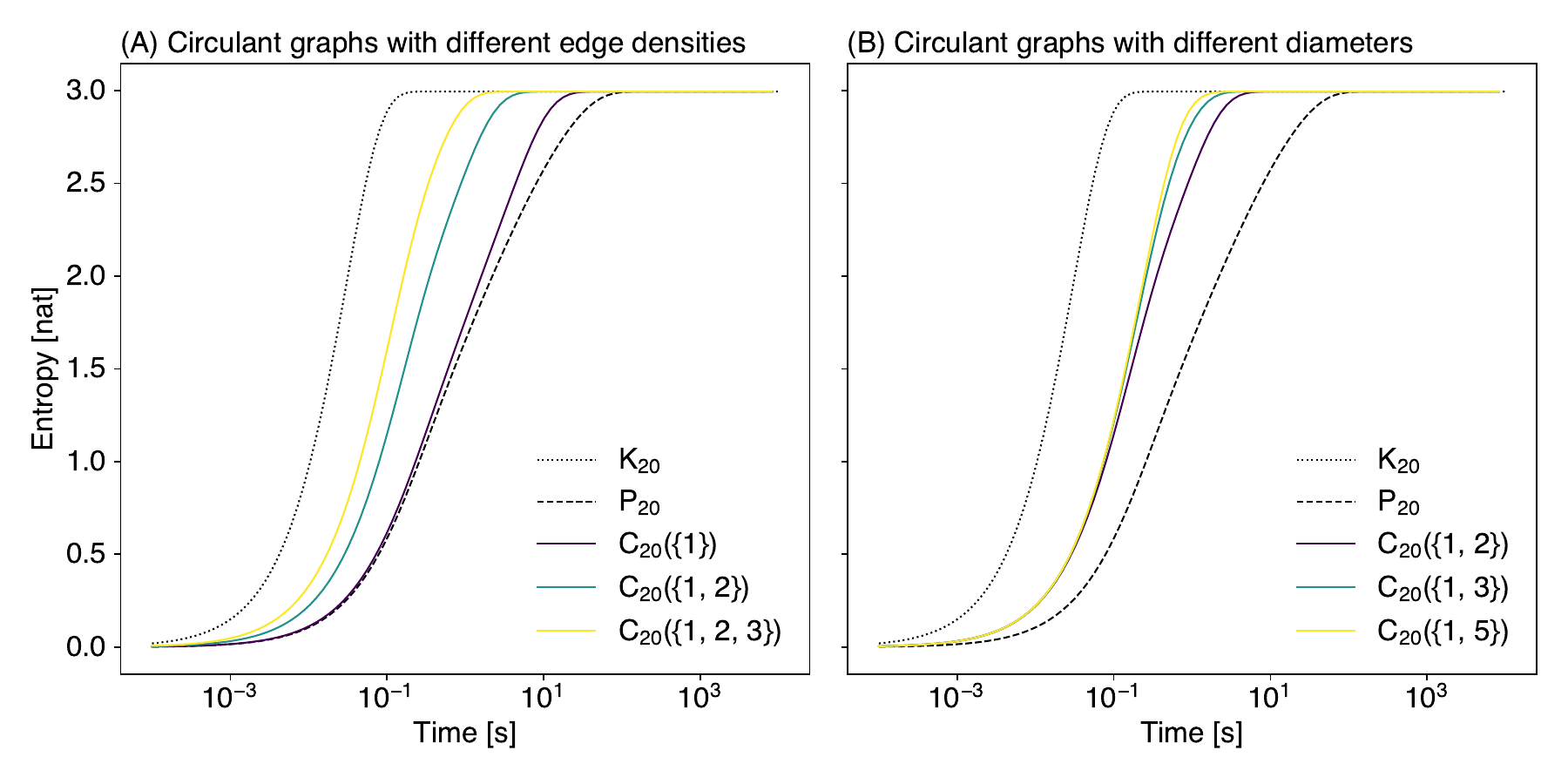}
  \caption{Conditional entropy for heat diffusion on a complete graph, path graph, and circulant graphs with different step-sets. The entropy curves are bounded from below by the path graph curve and above by the complete graph curve. (A) Three circulant graphs with increasing density. The density of $C_{20}(\{1\})$ is equal to $\frac{2}{19}$, the one of $C_{20}(\{1, 2\})$ is $\frac{4}{19}$, and the one of $C_{20}(\{1, 2, 3\})$ is $\frac{6}{19}$. The entropy grows faster in denser graphs. (B) Three circulant graphs with the same density but decreasing diameter. The diameter of $C_{20}(\{1\})$ is equal to $10$, the one of $C_{20}(\{1, 2\})$ is $5$, and the one of $C_{20}(\{1, 2, 3\})$ is $4$.The entropy grows faster as the diameter decreases.}
  \label{fig:circulant_curves}
\end{figure}

\subsection{Asymptotic Behaviour} \label{subsec:asymptotic_behaviour}
For a generic (random) graph, there is not necessarily such an explicit formula for its conditional entropy curve; however, we can say more about its asymptotic behaviour. We first derive the asymptotic value, which on connected graphs depends only on the size of the graph, and then describe how much time is required for the system to approach it.

We start by noting that the rows of the limit transition probability matrix  $\lim_{t \to \infty} \T(0,t)$ of an ergodic Markov chain are given by its stationary distribution $\Tilde{\matr{\pi}}$.
Given that the conditional entropy averages the row entropies with respect to the initial condition $p(0)$, this implies that 
its asymptotic is equal to the entropy of $\Tilde{\matr{\pi}}$, hence is independent of the initial distribution, i.e.
    \begin{equation}
        \forall p(0), q(0) :  H(p(t) | p(0)) = H(q(t) | q(0)).
    \end{equation}
For heat diffusion in a connected graph, this value  $\lim_{t \to \infty} H(Z (t) | Z(0))$ is equal to $\log(N)$ independently of the network's structure.
In the \cref{appendix:RW} we show that using the random walk diffusion dynamic leads to a different asymptotic value, determined by the degree distribution (\cref{prop_asymptot_RW}).\\
The case of multiple connected components can then be understood by applying the results for connected components to each component separately.
\begin{proposition}
    Consider a graph with vertex set $V=\bigcup_{k=1}^l V_k$ partitioned into disjoint connected components with vertex subsets $V_1, \dots, V_l$. Then we have:
    \begin{equation} \label{eq:heatlimit_disjoint}
        \lim_{t \to \infty} H(Z (t) | Z(0)) = \sum_{k=1}^l \log(|V_k|)\sum_{j \in V_k}p_j(0),
    \end{equation}
and if the distribution of the initial distribution $p(0)$ is uniform,
    \begin{equation}
        \lim_{t \to \infty} H(Z (t) | Z(0)) = \sum_{i=1}^l \frac{|V_i|}{N} \log(|V_i|).
    \end{equation}
\end{proposition}
\begin{proof}
    Heat can only diffuse between nodes in the same connected component. Let us denote by $V(i)$ the connected component to which node $i$ belongs. Then we have:
    \begin{align*} 
        \lim_{t \to \infty} H(Z (t) | Z(0)) &= - \sum_{i=1}^N p_i(0) \sum_{j=1}^N  \frac{\delta_{V(i) = V(j)}}{\sum_{k=1}^N\delta_{V(i) = V(j)}} \log \left(\frac{\delta_{V(i) = V(j)}}{\sum_{k=1}^N\delta_{V(i) = V(j)}} \right) \\
        &= - \sum_{i=1}^N p_i(0) \sum_{j \in V(i)}  \frac{1}{|V(i)|} \log \left(\frac{1}{|V(i)|} \right) =
        \sum_{i=1}^N p_i(0) \log(|V(i)|) \\
        &= \sum_{k=1}^l \log(|V_k|)\sum_{j \in V_k}p_j(0).
    \end{align*}
    If we set $Z(0)$ to be uniformly distributed, we then get:
    \begin{align*}
         \lim_{t \to \infty} H(Z (t) | Z(0)) &= \sum_{k=1}^l \log(|V_k|)\sum_{j \in V_k} \frac{1}{N} = \sum_{i=1}^l \frac{|V_i|}{N} \log(|V_i|).
    \end{align*}
\end{proof}

Regarding how much time it takes to get to stationarity, we remind the readers of a known fact: for ergodic finite-state Markov chains coordinate-wise convergence of the probability vectors implies convergence of the $p$-norms in finite dimensions, so  if $G$ is a connected graph, then
    \begin{equation}
        \forall p \in [1, \infty]:\lim_{t \to \infty} \Vert \matr{\pi}^{\intercal}- \p(t)^{\intercal}\Vert_p  = 0.
    \end{equation}
There are multiple ways to describe how much time it takes to get close to the asymptotic level \cite{durrett_dynamics_2025, levin_markov_2017}. For instance, since the transition matrices of our system are diagonalizable, we can estimate the deviation from stationarity as 
\begin{align*}
\matr{\pi}^{\intercal} - \p(t)^{\intercal} &=  \matr{\pi}^{\intercal} - \left( \matr{\pi}^{\intercal} + \p(0)^{\intercal} \sum_{i=2}^{N} e^{-\lambda_i t} \matr{v}^{(i)} \matr{v}^{(i)^\intercal} \right)  \\
&= \p(0)^{\intercal} \sum_{i=2}^{N} e^{-\lambda_i t} \matr{v}^{(i)} \matr{v}^{(i)^\intercal}  
=  \p(0)^{\intercal} e^{-\lambda_2 t} \matr{v}^{(2)} \matr{v}^{(2)\intercal} + \p(0)^{\intercal} \sum_{i=3}^{N} e^{-\lambda_i t} \matr{v}^{(i)} \matr{v}^{(i)^\intercal}.
\end{align*}
Taking the limit, we get:
\begin{equation} \label{eq:mixing_time}
    \lim_{t \to \infty} \matr{\pi}^{\intercal} - \p(t)^{\intercal} \approx  e^{-\lambda_2 t} \p(0)^{\intercal} \matr{v}^{(2)} \matr{v}^{(2)^\intercal}.
\end{equation} 

This shows that the rates of decay are given by the Laplacian eigenvalues, which depend on the graph density, and that the spectral gap, $\lambda_2$, gives the slowest rate and determines the mixing time.

\subsection{Upper and Lower Bounds for the Evolution of the Conditional Entropy}

Van Mieghem \cite[art. 163]{van_mieghem_graph_2023} explains that Laplacian eigenvalues are non-decreasing if edges are added to a graph, implying that denser graphs have shorter mixing time. We give a simpler proof of this result and use it to argue that the eigenvalues of the complete graph's Laplacian are maximal in the following sense:
\begin{proposition} \label{prop:edge_mono}
    Let $G=(V,E)$ be a subgraph with $N$ nodes of $G'=(V, E \bigcup (u,v) )$ obtained by removing an edge from it. Consider the ordered Laplacian spectra $\{ \lambda_i^G \}_{1 \leqslant i \leqslant N}$, $\{ \lambda_i^{G'} \}_{1 \leqslant i \leqslant N}$ of $G$ and $G'$ and the Laplacian spectrum $\{ \lambda_i^K \}_{1 \leqslant i \leqslant N}$ of the complete graph $K_N$ of the same size. Then we have that
    \begin{equation}
        \forall i : \lambda_i^G \leqslant \lambda_i^{G'}  \leqslant \lambda_i^{K_N}.
    \end{equation} 
\end{proposition}
\begin{proof}
    We first address the inequality on the left. If we denote by $e_u$ and $e_v$ the basis vector corresponding to nodes $u$ and $v$, then the Laplacian $\Lapl'$ of $G'$ can be expressed as the Laplacian $\Lapl$ of $G$ plus a perturbation $\Lapl_e \coloneqq(e_u - e_v) (e_u - e_v)^{\intercal}$. Notice that $\Lapl_e$ has rank equal to one and the only non-zero eigenvalue is equal to two.
    Weil's inequality \cite{weyl_asymptotische_1912, helmke_eigenvalue_1995} states that
    \begin{equation}
        1 \leqslant j \leqslant i \leqslant j \leqslant N: \lambda_{i-j+1} (\Lapl) + \lambda_j (\Lapl_e) \leqslant \lambda_{i}(\Lapl + \Lapl_e)  \leqslant \lambda_{N-k+i} (\Lapl) + \lambda_k (\Lapl_e).
    \end{equation}
    By choosing $j=1$, the left-side inequality leads to $\lambda_i(\Lapl) \leqslant \lambda_i(\Lapl')$. Furthermore, every graph is a subgraph of the complete graph; thus, by monotonicity, the second desired inequality is also proven.
\end{proof} 
It follows immediately that:
\begin{corollary}
    Let $G$ be a graph of size $N$. Then:
    \begin{equation}
        2 \leqslant i \leqslant N: \lambda_i(G) \leqslant N = \lambda_2(K_N).
    \end{equation}
\end{corollary}
\begin{proof}
The previous proposition says that $ 2 \leqslant i \leqslant N: \lambda_i(G) \leqslant \lambda_i(K_N)$  and in \cref{lemma:spectrumKN} we have seen that $2 \leqslant i \leqslant N: \lambda_i(K_N) = \lambda_2(K_N) = N$.   
\end{proof}
This establishes the complete graph as the structure upon which diffusion unfolds the fastest. This is intuitively clear, as more edges correspond to more pathways for diffusion. The same argument can not be used to establish an optimal slowest structure: trees have minimal eigenvalues, but no tree has uniformly smaller ones, as one can see by confronting the spectra of the path and star graphs. \\
However, we can argue that at least asymptotically, the path graph will be the one that reaches the asymptotic level the slowest.

\begin{proposition}
    Consider a connected graph $G$ with $N$ nodes. Then its algebraic connectivity is bigger than or equal to that of the path graph $\lambda_2(P_N) = 2 \left( 1-\cos \left( \frac{\pi}{N} \right) \right) = 4\sin^2 \left( \frac{\pi}{2N} \right)$.
\end{proposition}

\begin{proof}
    Recall from \cref{eq:spectrum_path} that we know that the algebraic connectivity of $P_N$ is
    \begin{equation}
        \lambda_2 = 2 \left( 1-\cos \left( \frac{\pi}{N} \right) \right) 
              = 4\sin^2 \left( \frac{\pi}{2N} \right).
    \end{equation}  
    Consider a spanning tree $T$ of $G$. If $T$ is isomorphic to $P_N$, then by the edge-monotonicity principle stated in \cref{prop:edge_mono}, the algebraic connectivity of $G$ must be bigger than or equal to that of $P_N$. Otherwise, \cite[4.3]{fiedler_algebraic_1973} states that among trees with $N$ nodes, the $P_N$ has minimal algebraic connectivity. Hence we conclude that $\lambda_2(P_N) \leqslant \lambda_2(T) \leqslant \lambda_2(G)$.
\end{proof}
We can use these results on the mixing time to argue that the complete and path graph are good representatives of fast and slow diffusion structures to use as references for other graphs when plotting their conditional entropy curves.
\\
Finally, we end this section with a final result linking the conditional entropy to the total variation of the rows of the transition matrix from the stationary distribution. 
\begin{proposition}
    For every $p \geqslant 1$:
    $\log(N) - H(Z(t) | Z(0)) \geqslant \sum_i\frac{p_i(0)}{2} \|\p^{(i)}- \matr{\pi} \|_p^2.$
\end{proposition}
\begin{proof}
    We have seen in the proof of \cref{Thm:2ndlaw} that $
    \log(N) - H(\p^{(i)})\;=\; \dkl(\p^{(i)}\|\matr{\pi})$.
    As a direct consequence of Pinsker inequality (\cref{prop:Pinsker}) we have for every row distribution $\p^{(i)}$,
    \begin{equation*}
        \log(N) - H(\p^{(i)})\;=\; \dkl(\p^{(i)}\| \matr{\pi}) \geqslant \frac{1}{2} \| \p^{(i)}- \matr{\pi} \|_1^2.
    \end{equation*}
    Since for any vector $\matr{x}$, $\| \matr{x} \|_1 \geqslant \| \matr{x} \|_p$ for every $p \geqslant 1$, this yields
    \begin{equation*}
        \log(N) - H( \p^{(i)}) \geqslant \frac{1}{2} \| \p^{(i)}- \matr{\pi} \|_p^2,
    \end{equation*} which tells us that $\log(N) - H(Z(t) | Z(0)) \geqslant \sum_i\frac{p_i(0)}{2} \| \p^{(i)}- \matr{\pi} \|_p^2$.
\end{proof}

\subsection{Simulations on Random Networks} \label{sec:Simulations}
 In \cref{subsec:Thermodynamics} we showed local properties of the conditional entropy of diffusion dynamics (conservation and monotonicity). For certain classes of graphs, high regularity enabled us to compute the conditional entropy evolution curves explicitly (\cref{subsec:extact_results}). In  \cref{subsec:asymptotic_behaviour} we derived general results describing the asymptotic behaviour of conditional entropy, and we showed that the conditional entropy curves of the full and path graphs of a certain size bound asymptotically the region where we would expect to observe the conditional entropy curve of any graph of the same size. Now, we experimentally capture the properties of the evolution of conditional entropy in diffusion on two classes of random graphs that are prominent in the literature on complex networks: Watts-Strogatz and Erdős-Rényi graphs \footnote{The code to reproduce the simulation is available at \url{https://github.com/samuelkoovely/Evolution-of-Conditional-Entropy-for-Diffusion-Dynamics-on-Graphs}.}.
\\ \\
Watts-Strogatz networks \cite{watts_collective_1998, watts_small_1999} are obtained by randomly rewiring a small number of edges in circulant graphs where each node is connected precisely to the $k$ previous and $k$ subsequent ``neighbours". The randomly rewired edges create shortcuts in the networks, giving them the so-called small-world property, i.e., their diameters are proportional to the logarithm of the number of nodes. More formally, a Watts-Strogatz network is a sample of the Watts-Strogatz model $WS_N(\{1, \dots, k\}, p)$ obtained by taking the circulant graph $C_N(\{1, \dots, k\})$ and rewiring each edge with probability $p$ uniformly at random, while assuring that no self-loop or multi-edge is generated. 
\\
In \cref{fig:WS_ER_vs_circulant} (A), we simulate heat diffusion on 10 Watts-Strogatz graphs with 100 nodes with a uniform initial condition and compute their conditional entropy. We then illustrate the comparison between the entropy of simulated networks and their non-perturbed circulant graph counterpart. Finally, we add the complete and path graphs of the same size as references.
Initially, the entropy curves on Watts-Strogatz networks grow similarly to those of the reference circulant graph. Later, the circulant graph's curve falls behind. This fact is explained by the presence of shortcuts in Watts-Strogatz networks that accelerate diffusion. Indeed, the entropy grows even faster as we rewire more edges, creating more shortcuts. \\ \\
Erd\"os-R\'enyi graphs (ER graphs) are generated by a random model that, given a predetermined number of nodes $N$, include any possible edge independently with a fixed probability $p$.
In \cref{fig:WS_ER_vs_circulant} (B), we sample 10 ER graphs and compare them with circulant graphs with the same density but varying diameters.
Initially, these curves grow similarly, then graphs with higher diameters fall behind. The ER curve reaches the asymptotic level faster than the circulant counterpart with the same diameter, because of the higher number of nodes at distance two.

\begin{figure}[htbp]
  \centering
  \label{fig:WS_ER_vs_circulant}\includegraphics[width=0.95\linewidth]{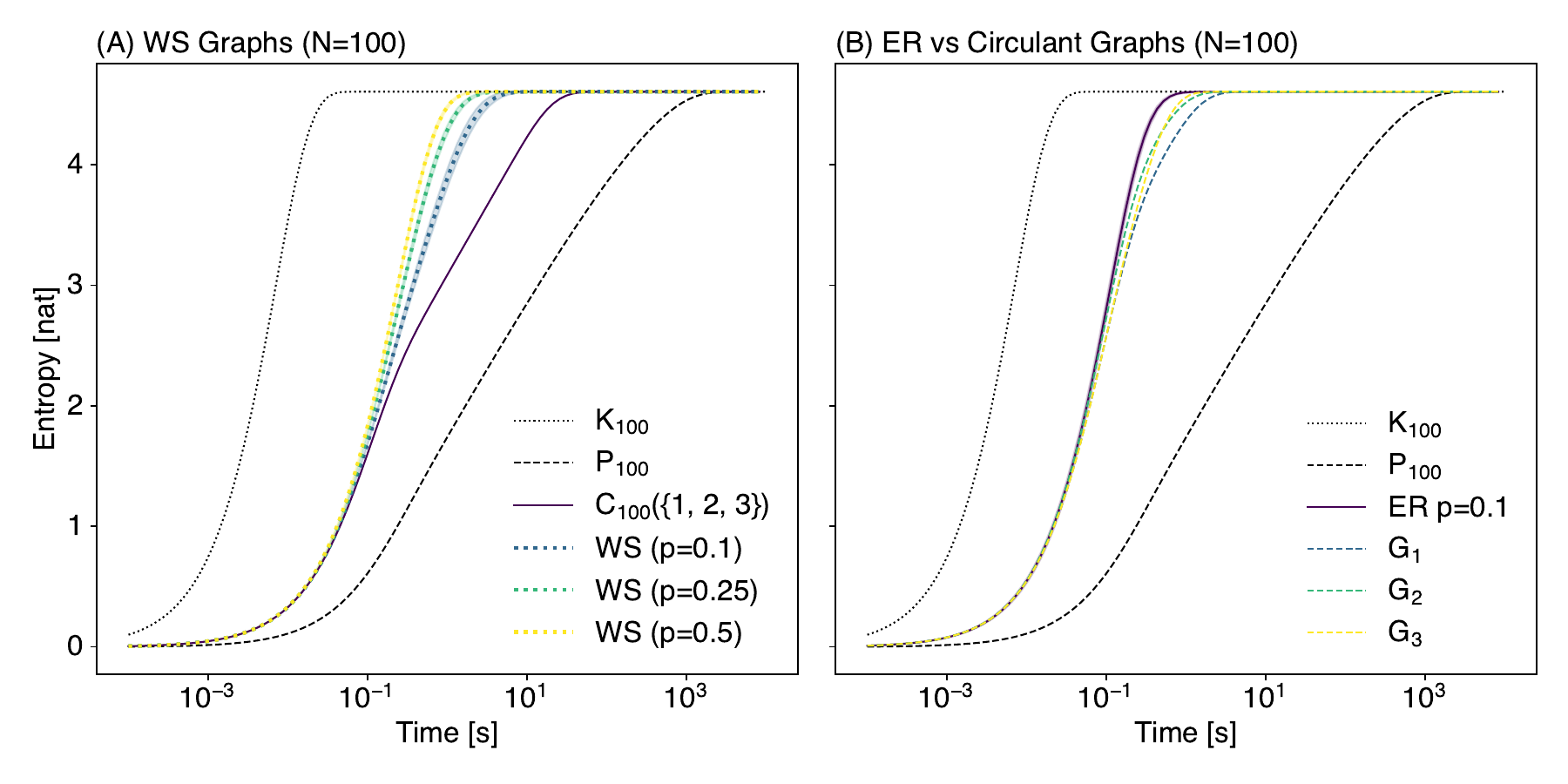}
  \caption{Comparison of Watts-Strogatz and ER graphs with circulant matrices sharing similar structural properties. For all stochastic models, we sample $10$ graphs of size $100$. (A) We use as reference $C_{100}(S)$, with $S = \{ 1,2,3 \}$. We use it as a skeleton for obtaining three Watts-Strogatz network models with increasing probability of edge rewiring: $p= 0.1, 0.25, 0.5$. The plot shows that by rewiring more edges, we raise the speed at which the conditional entropy increases. For high numbers of rewired edges, the variance of the entropy curve is low (shaded area indicates $\pm 1$ standard deviation). (B) We compare the conditional entropy averaged on a set of ER graphs to that of circulant graphs $G_1 = C_{100}(\{ 1,2,3,4, 13 \}), G_2 = C_{100}(\{ 1,3,5,11,13 \}), G_3 = C_{100}(\{ 1,10,11,12, 13 \})$ having a similar density to that of the ER graphs, but having other different structural properties. Indeed, the initial shape of the curve is similar because the number of neighbors is similar (by design). On the other hand, the conditional entropy on ER graphs then increases faster than on the circulant graphs because of the higher number of nodes at distance two compared to comparable circulant graphs. $G_1$ and $G_2$ have both diameters equal to $5$, but in the latter, nodes have fewer nodes at distances $4$ and $5$, whereas more at distances $2$ and $3$, leading to a faster growth of conditional entropy. In $G_3$, nodes have fewer nodes at distances $2$ compared to $G_2$, but more at distances $3$ and $4$, which is its diameter. Therefore, its entropy grows more slowly initially, but then catches up and overtakes that of $G_2$.}
\end{figure}

\subsection{Mean-field Approximation for Erdős-Rényi Graphs}
For ER graphs, we suggest a deterministic approximation. We treat the case where ER graphs are connected, which is a reasonable assumption in the supercritical regime where the average number of neighbors $c$ is bigger than 1. In this case, most of the nodes belong to the largest connected component of the graph. In particular, we know that the proportion of nodes that belong to this giant component in an ER graph is given by $S = 1 + \frac{W(-c e^{-c})}{c}$, where we take the principal branch of the Lambert $W$-function \cite[Chapter 11.5]{newman_networks_2018}.
This number grows fast as a function of $c$ to the asymptotic value 1, justifying our assumption.
% Mean-Field Approximation for Diffusion on ER Graphs
\begin{figure}[htb]
  \centering
  \label{fig:mean_field}\includegraphics[width=0.95\linewidth]{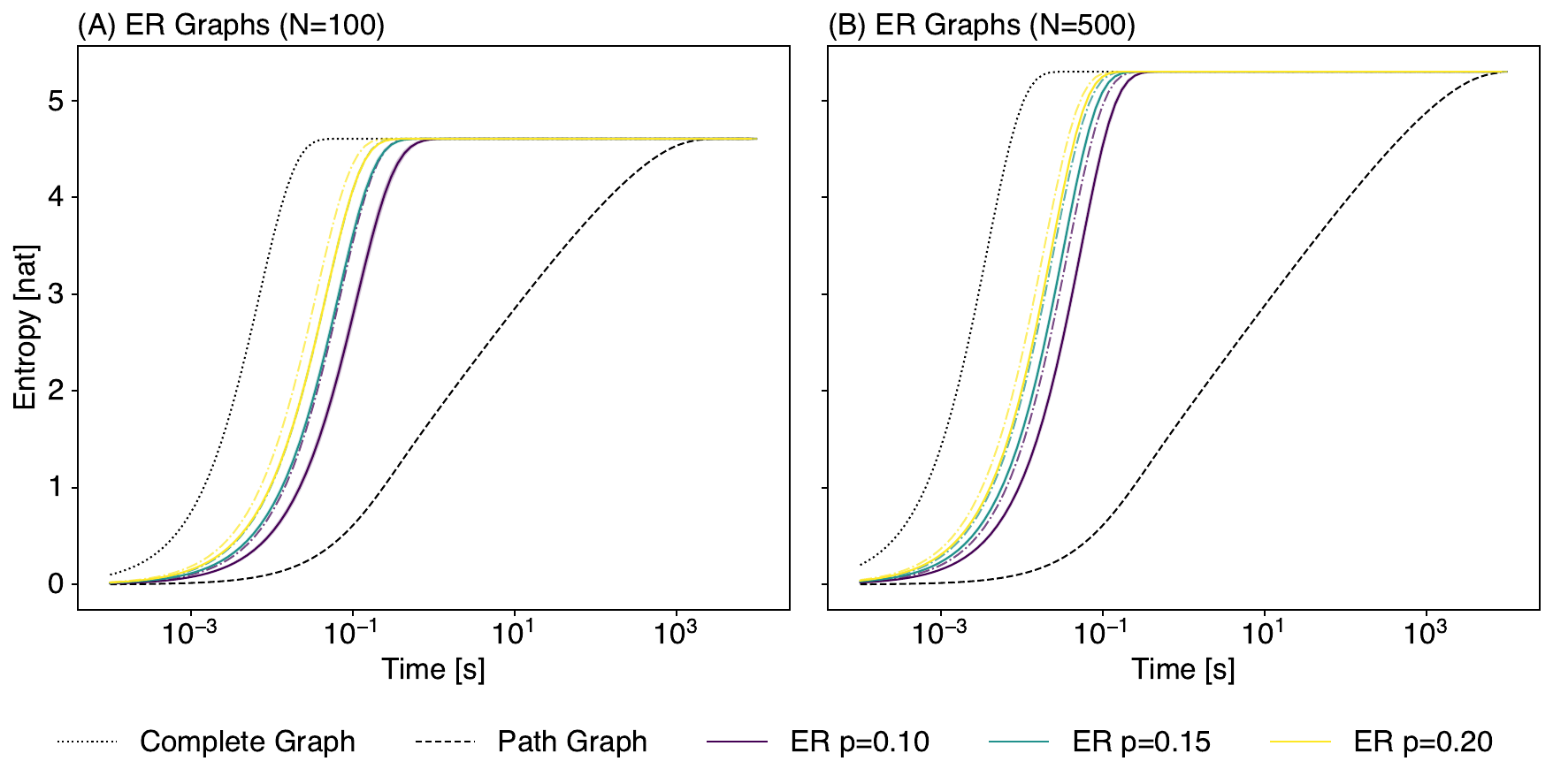}
  \caption{Evolution of the conditional entropy over samples of 5 ER graphs (dotted lines)  with three different density values plotted together with their mean-field approximation (dashed-dotted lines). (A) shows graphs of size 100, (B) graphs of size 500. The mean-field approximation overestimates entropy values; however, one can see that it becomes more accurate as the density of the graph and the size of the ER graph increase.}
\end{figure}
% Mean degree and average Laplacian
In the ensemble mean-field approximation, each node is statistically equivalent. The expected degree of each node is approximately $pN$.
The mean-field Laplacian is then $\langle \Lapl \rangle \coloneqq p(N\Id_N - \matr{J}_N)$ and its entries are
\begin{equation}
\langle L_{i,j} \rangle = \begin{cases}
pN & \text{ if } i = j, \\
-p & \text{ if } i \neq j.
\end{cases}
\end{equation}
% Exponential of the mean-field Laplacian
We compute the average transition operator $e^{-\langle \Lapl \rangle t} = e^{-p(N \Id - \matr{J})t} = e^{-pNt} e^{p t \matr{J}}$.
We have $e^{p t \matr{J}} = \matr{I} + \frac{\matr{J}}{N} \left(e^{pNt} - 1 \right)$, therefore
\begin{equation}
e^{-\langle \Lapl \rangle t} = e^{-pNt} \Id + \frac{1 - e^{-pNt}}{N} \matr{J}.
\end{equation}
Which is similar to the result for the complete graph (\cref{eq:T_complete}), with an additional factor $p$ in the argument of the exponentials. In entry-wise form,
\begin{equation}
e^{-\langle \Lapl \rangle t}_{i,j} = \begin{cases}
a \coloneqq e^{-pNt} + \frac{1 - e^{-pNt}}{N} & \text{ if } i = j, \\
b \coloneqq \frac{1 - e^{-pNt}}{N} & \text{ if } i \neq j.
\end{cases}
\end{equation}
% Shannon entropy per row
Hence, we obtain an explicit formula for the mean-field conditional entropy of heat diffusion on a connected ER graph, independent of the initial distribution:
\begin{equation}
\forall p(0): H(p (t) | p(0)) = -\left( a \ln a + (N-1) b \ln b \right).
\end{equation}
We show the quality of this mean-field approximation in \Cref{fig:mean_field}.
In both panels we sample 5 ER graphs with three different density values and compare them to their mean-field approximation. Panel (A) shows graphs of size 100, panel (B) shows graphs of size 500. In both plots, the mean-field approximation overestimates entropy values, however, it becomes more accurate as we increase the density of the graph. Moreover, the approximations shown on the right panel are more accurate than those on the left, suggesting that the mean-field approximation is more accurate for larger graphs.

\section{Conclusions}
We begin our conclusions by summarizing and contextualizing our findings. We conclude by offering some insights into the potential impact of this work and foreseeable research directions that stem from it.

We introduced a thermodynamic formalism for networks by combining diffusion dynamics and conditional entropy. The information-theoretical properties of Markov chains allowed us to investigate monotonicity of diffusion dynamics on graphs. Based on this discussion, we decided to focus on heat diffusion, for which we derived an asymptotic formula. We provided the formulas and proofs of the eigendecomposition of the heat transition matrices for complete, path, and circulant graphs, which we use for fast computation of their conditional entropy curves. Our analysis of mixing times partially justifies the use of the complete and path graphs as extremal cases to bound the entropy of other graphs of the same size. We used circulant graphs to gain a deeper understanding of conditional entropy in terms of edge density and the network's diameter. We also showed experimentally that the conditional entropy of heat increases more rapidly on random structures such as Watts-Strogatz and ER graphs compared to their circulant counterparts. Finally, we proposed a mean-field approximation for ER graphs and demonstrated its quality experimentally.

Despite the popularity of random walk and diffusion-based methods on graphs (e.g., \cite{newman_measure_2005,pons_computing_2006,rosvall_maps_2008,delvenne_stability_2010, lambiotte_random_2014,masuda_random_2017,carletti_random_2020,bovet_flow_2022}), there has been so far limited attention to studying the entropic properties of such processes on graphs. Some of the earlier works include Refs. \cite{latora_kolmogorov-sinai_1999, gomez-gardenes_entropy_2008}, while Refs. \cite{villegas_laplacian_2022, villegas_laplacian_2023, nartallo-kaluarachchi_broken_2024} are fruits of the recent revival in interest in the subject. This work is therefore part of the ongoing effort to establish mathematical frameworks for information theory, a unifying language to describe stochastic dynamics on networks. Different notions of entropy \cite{yang_unified_2020} possess distinct properties; our work highlights that, from the perspective of conditional entropy, heat diffusion on graphs behaves similarly to heat diffusion on a solid medium in the real world.
By analogy, diffusion on heterogeneous materials could be thought of as inspiration for studying the conditional entropy of heat diffusion on temporal and multilayer networks. Indeed, this work lays the groundwork for the development and understanding of physics-inspired methods in applications such as detecting change points and periodic patterns in the evolution of temporal networks.

The derivation of a thermodynamic formalism in a non-physical system, such as graphs, is also interesting from a philosophical perspective. Indeed, by expanding on a point raised by Cover in \cite{halliwell_physical_1996}, the fact that the information-theoretical properties of certain Markov chains are similar to those of physical heat offers an interesting perspective on the notoriously poorly understood phenomenology of the second law of thermodynamics. Indeed, one could argue that this fact points towards a Markovian description of thermodynamic systems (arguably the closest thing to determinism in a stochastic world), contextualizing the irreversible growth of entropy not as an a priori law of the universe, but as the product of the stochasticity that characterizes it. Future research in this direction could therefore explore other physics-inspired entropy quantities, such as the Wehrl Entropy \cite{wehrl_relation_1979} or, taken from information theory, the R\`enyi entropy \cite{renyi_measures_1961}. However, one should expect difficulties in the mathematical derivations. On this note, a significant step forward in the theory would be the derivation of formulas that describe the expected average behavior in random graphs, going beyond simulations such as the one we presented for Watts-Strogatz and ER graphs.

\appendix
\section{Data Processing Inequality} \label{app:DPI}

\begin{lemma}[Data-processing inequality for a Markov kernel]
\label{lem:DPI_markov_kernel}
Consider two Markov chains $\{ X(t) \}_{t \in \mathbb{R}^+}$ and $\{ Y(t) \}_{t \in \mathbb{R}^+}$ on the same state space and with the same transition matrix $\T(t_1,t_2)$ between times $t_1$ and $t_2$, with $t_1 \leq t_2$. Let $\p(t)$ and $\matr{q}(t)$ be the respective probability vectors at time $t$. Then:
\begin{equation}
\label{eq:DPI_markov_kernel}
    \dkl\!\left(\p(t_2)\,\|\,\matr{q}(t_2)\right)
    \leq
    \dkl\!\left(\p(t_1)\,\|\,\matr{q}(t_1)\right).
\end{equation}
\end{lemma}

\begin{proof}
By the chain rule for KL divergence (\cref{prop:chainKL}), we have
\begin{equation}
\label{eq:appendix_KL_chain_forward}
    \dkl\!\left(p(t_1,t_2)\,\|\,q(t_1,t_2)\right)
    =
    \dkl\!\left(\p(t_1)\,\|\,\matr{q}(t_1)\right)
    +
    \dkl\!\left(p(t_2\mid t_1)\,\|\,q(t_2\mid t_1)\right).
\end{equation}
We notice that, since both chains have the same transition probabilities,
\begin{equation*}
    p_j(t_2\mid X(t_1)=i)
    =
    q_j(t_2\mid Y(t_1)=i)
    =
    T_{ij}(t_1,t_2),
\end{equation*}
we get $\dkl\!\left(p(t_2\mid t_1)\,\|\,q(t_2\mid t_1)\right)=0$, hence
\begin{equation}
\label{eq:appendix_joint_KL_equals_initial_KL}
    \dkl\!\left(p(t_1,t_2)\,\|\,q(t_1,t_2)\right)
    =
    \dkl\!\left(\p(t_1)\,\|\,\matr{q}(t_1)\right).
\end{equation}
On the other hand, applying the chain rule in the opposite order gives
\begin{equation}
\label{eq:appendix_KL_chain_backward}
    \dkl\!\left(p(t_1,t_2)\,\|\,q(t_1,t_2)\right)
    =
    \dkl\!\left(\p(t_2)\,\|\,\matr{q}(t_2)\right)
    +
    \dkl\!\left(p(t_1\mid t_2)\,\|\,q(t_1\mid t_2)\right).
\end{equation}
By non-negativity of KL divergence, $\dkl\!\left(p(t_1\mid t_2)\,\|\,q(t_1\mid t_2)\right) \geq 0$; so, by combining \cref{eq:appendix_joint_KL_equals_initial_KL,eq:appendix_KL_chain_backward}, we obtain
$\dkl\!\left(\p(t_2)\,\|\,\matr{q}(t_2)\right) \leq\dkl\!\left(\p(t_1)\,\|\,\matr{q}(t_1)\right)$,
which proves the claim.
\end{proof}

\section{Computations for Complete graphs} \label{appendix:approx_complete}
We write the explicit formula of the conditional entropy of heat diffusion on complete graphs.
\begin{proof}
Let $Z(0)$ be distributed according to any initial distribution $\p(0)$, then
\begin{align*}
     H(Z(t) | Z(0)) &= -\sum_{i} p_i(0) \left[ \sum_{j \neq i} \frac{1}{N} \left(1 - e^{-N  t} \right) \log \left( \frac{1}{N} \left(1 - e^{-N  t} \right) \right) + \right. \\
    &+ \left. \frac{1}{N} \left(1 + (N-1)e^{-N  t} \right) \log \left( \frac{1}{N} \left( 1 + (N-1)e^{-N  t} \right) \right) \right] =\\
    &= -  \left( \frac{N-1}{N} - \frac{N-1}{N} e^{-N  t} \right) \left[ \log(\frac{1}{N}) + \log(1-e^{-N  t})\right] - \\
    &- (\frac{1}{N} + \frac{N-1}{N}) e^{-N  t} \left[ \log(\frac{1}{N}) + \log(1 + (N-1) e^{-N  t})\right] = \\
    &= - \left\{ \log(\frac{1}{N}) + \frac{N-1}{N} e^{- \frac{N  t}{N-1}} \log \left( \frac{1 + (N-1) e^{-\frac{N  t }{N-1}}}{1- e^{-N  t}} \right)+ \right. \\
    &+  \left. \frac{N-1}{N}\log(1-e^{-N  t}) + \frac{1}{N} \log(1+ (N-1) e^{-N  t}) \right\},
\end{align*}
    which leads to \cref{eq:exactHcomplete}.
\end{proof}

\section{Random Walk Diffusion} \label{appendix:RW}
Another diffusion dynamic that is prominent in the literature is random walk diffusion, whose intensity matrix is specified by the so-called random-walk Laplacian (denoted $\Lapl_{\mathrm{RW}}$) and is defined as
\begin{equation}
    (\Lapl_{\mathrm{RW}})_{ij}
    =
    \begin{cases} 
      1, & \text{if } i=j \text{ and } d_i>0, \\
      -\frac{1}{d_i}, & \text{if } \{v_i,v_j\}\in E,\ i\neq j, \\
      0, & \text{otherwise}.
   \end{cases}
\end{equation}
where $d_i$ denotes the degree of node $i$. \\
We can then write the corresponding Kolmogorov forward equation:
\begin{equation}
    \frac{d \p^{\intercal}(t)}{dt} = - \lambda \p^{\intercal}(t)\Lapl_{\mathrm{RW}} , \text{ with the solution } \p^{\intercal}(t) = \p^{\intercal}(0) e^{- \lambda \Lapl_{\mathrm{RW}}t},
\end{equation}
In connected graphs, for random walk diffusion, the process will converge to the stationary distribution $ \matr{\pi}^{\intercal}_{\mathrm{RW}} = [\frac{1}{d_1}, \frac{1}{d_2}, \dots, \frac{1}{d_N}] $.
\newcommand{\trwone}{\frac{1}{N} + \frac{N-1}{N} e^{- \frac{N}{N-1} t}}
\newcommand{\trwtwo}{\frac{1}{N} - \frac{1}{N}e^{-  \frac{N}{N-1} t}}
One can derive the formula for $\T_{\mathrm{RW}}$ on a complete graph immediately from \cref{eq:CTHeat_complete} by noticing that given the constant degree, random-walk diffusion is just a slowed-down version of heat diffusion where the rate of diffusion is equal to $\frac{1}{N-1}$ instead of $1$.
\begin{equation}\label{eq:CTRW_complete}
    T_{RW_{i,j}} (t)= \begin{cases}
   \trwone & \text{if } i = j, \\
    \trwtwo & \text{if } i \neq j.
\end{cases}
\end{equation}
From \cref{eq:CTRW_complete} we derive a formula for the entropy of random walk, analogously to \cref{eq:exactHcomplete}:
    \begin{equation} \label{eq:exactRWcomplete}
        \begin{aligned}
             H(Z_{\mathrm{RW}}(t) | Z_{\mathrm{RW}}(0)) &= \log(N) 
            - \frac{N-1}{N} e^{-\frac{N t}{N-1}} 
            \log\left(\frac{1 + (N-1) e^{-\frac{N t}{N-1}}}{1 - e^{-\frac{N t}{N-1}}}\right) \\
            &- \frac{N-1}{N} \log\left(1 - e^{-\frac{N t}{N-1}}\right) 
            - \frac{1}{N} \log\left(1 + (N-1) e^{-\frac{N t}{N-1}}\right).
            \end{aligned}
    \end{equation}
We provide a formula for the asymptotic value of random walk diffusion, which depends on the degree distribution of the nodes:
\begin{proposition} \label{prop_asymptot_RW}
    Let $\langle R \rangle$ indicate the expectation of a random variable $R$ with respect to the degree distribution $P_d$ of a graph. If that graph is connected, then the asymptotic value of the conditional entropy of random walk diffusion over it is
    \begin{equation} \label{eq_RWlimit}
         \lim_{t \to \infty} H(Z_{\mathrm{RW}}(t) | Z_{\mathrm{RW}}(0)) = \log(2M) - \frac{\langle d \log(d) \rangle}{\langle d \rangle},
    \end{equation}
    independently from the initial distribution $p(0)$.
\end{proposition}
\begin{proof}
    In the case of random walk diffusion, the asymptotic value of the transition matrix $\lim_{t \to \infty} T_{i,j}(0,t) $ is equal to $\frac{d_j}{2M}$.
    We therefore have that:
        \begin{align*}
            \lim_{t \to \infty} H(Z_{\mathrm{RW}}(t) | Z_{\mathrm{RW}}(0)) &=  H( \lim_{t \to \infty} Z_{\mathrm{RW}}(t) | Z_{\mathrm{RW}}(0)) = H(\matr{\pi}_{\mathrm{RW}}) \\
            &= - \sum_i \frac{d_i}{2M} \log(\frac{d_i}{2M}) \\
            &= -  \frac{1}{2M} \sum_i (d_i \log(d_i) - d_i \log(2M)) \\
            &= - \frac{1}{2M} [(\sum_i d_i \log(d_i)) - 2M \log(2M)] \\
            &= \log(2M) - \frac{1}{2M} \sum_j d_j \log(d_j)
        \end{align*}
    By switching to degree classes, we get:
    \begin{equation*}
        \lim_{t \to \infty} H(Z_{\mathrm{RW}}(t) | p) = \log(2M) - \frac{N}{2M}\sum_d P_d (d)  d \log(d) =  \log(2M) - \frac{ \langle d \log(d) \rangle}{\langle d \rangle}.
    \end{equation*}
\end{proof}
Given that the probability distribution over $N$ elements is maximal for the uniform distribution, we have that the limit value in  \cref{eq_RWlimit} is smaller or equal to the limit value for heat diffusion.

\section*{Acknowledgments}
We would like to acknowledge Thomas Lehéricy and Andrea Ulliana for the interesting discussions about the subject.

\bibliographystyle{siam}  
\bibliography{refs}  

\end{document}